\documentclass[12pt,letterpaper]{amsart}
\usepackage{geometry}
\usepackage{hyperref}
\usepackage{mathpazo,setspace,xcolor}
\newtheorem{theorem}{Theorem}
\newtheorem*{theorem*}{Theorem}
\newtheorem{proposition}{Proposition}
\newtheorem{lemma}{Lemma}
\newtheorem{corollary}{Corollary}
\theoremstyle{remark}

\newcommand{\C}{\mathbb{C}}
\newcommand{\D}{\Omega}
\newcommand{\ep}{\varepsilon}
\newcommand{\Dc}{\overline{\Omega}}
\newcommand{\dbar}{\overline{\partial}}

\onehalfspace

\title{On a theorem of Bishop and commutants of 
		Toeplitz operators in $\mathbb{C}^n$}
\author{S\"{o}nmez \c{S}ahuto\u{g}lu}
\author{Akaki Tikaradze}
\email{Sonmez.Sahutoglu@utoledo.edu, Akaki.Tikaradze@utoledo.edu}
\address{University of Toledo, Department of Mathematics \& Statistics, 
Toledo, OH 43606, USA}

\thanks{The second author is supported in part by the University of Toledo 
	Summer Research Awards and Fellowships Program.}

\date{\today}
\begin{document}

\begin{abstract}
We prove an approximation theorem on a class of domains in 
$\mathbb{C}^n$ on which the $\overline{\partial}$-problem 
is solvable in $L^{\infty}$. Furthermore, as a corollary, we obtain a 
version of the Axler-\v{C}u\v{c}kovi\'c-Rao Theorem in higher dimensions.
\end{abstract}
\maketitle

Let $\D$ be a domain in $\C^n$ and $\phi$ be a complex-valued function 
on $\D$. Let $H^{\infty}(\D)$ and $H^{\infty}(\D)[\phi]$ denote the set of 
bounded holomorphic functions on $\D$ and the algebra generated by 
$\phi$ over $H^{\infty}(\D)$, respectively. In 1989, Christopher Bishop 
proved the following approximation theorem 
(see \cite[Theorem 1.2]{Bishop89}).

\begin{theorem*}[Bishop]
Let $\Omega$ be an open set in $\mathbb{C}$ and $f$ be a bounded 
holomorphic function on $\D$ that is non-constant on every connected 
component of $\Omega$. Then $H^{\infty}(\Omega)[\overline{f}]$ is 
dense in $C(\overline{\Omega})$ in the uniform topology.
\end{theorem*}

In the same paper, Christophe Bishop also proved a stronger approximation 
result, \cite[Theorem 1.1]{Bishop89}, on a more restrictive class of domains 
on which $\overline{f}$ is only assumed to be a non-holomorphic harmonic 
function. Such a result for the unit disc goes back to  Sheldon Axler and 
Allen Shields \cite{AxlerShields87}.  Recently, Guangfu Cao gave an incorrect 
statement \cite[Theorem 5]{Cao08} in an attempt to give a higher dimensional 
version of Bishop's Theorem. Alexander Izzo and Bo Li \cite[pg 246]{IzzoLi13} 
noticed that the statement is incorrect. H{\aa}kan  Samuelsson and  
Erlend Wold in 	\cite[Theorem 1.3]{SamuelssonWold12} proved a partial 
extension of Bishop's Theorem for pluriharmonic functions and  $C^1$-smooth 
polynomially convex domains in $\C^n$.

This article is motivated by these papers and is an attempt to contribute  
an approximation theorem akin to Bishop's Theorem on domains in $\C^n$. 
We are not able to generalize Bishop's theorem to $\C^n$ and this is 
still an open problem. However, we prove approximation results under some 
restrictions on the functions and the domains. Furthermore, we apply our 
results to prove a  version of the Axler-\v{C}u\v{c}kovi\'c-Rao 
Theorem \cite{AxlerCuckovicRao00} in higher dimensions.

To present our first result we need to make some definitions. 
Let $\D\subset \C^n$ be a pseudoconvex domain and $CL^{\infty}_{(0,q)}(\D)$ 
denote the set of $(0,q)$-forms with coefficient functions 
that are $C^{\infty}$-smooth and bounded on $\D$. That is, 
$CL^{\infty}_{(0,q)}(\D)= L^{\infty}_{(0,q)}(\D)\cap C^{\infty}_{(0,q)}(\D)$. 
We call $\D$ a \textit{$L^{\infty}$-pseudoconvex}  domain if for 
$1\leq q\leq n$,  and $f\in CL^{\infty}_{(0,q)}(\D)$  such that $\dbar f=0$ 
there exists $g\in L^{\infty}_{(0,q-1)}(\D)$ such that $\dbar g=f$. 

The class of $L^{\infty}$-pseudoconvex domains include the products of 
$C^2$-smooth bounded strongly pseudoconvex domains 
\cite{SergeevHenkin80}, smooth bounded pseudoconvex finite 
type domains in $\C^2$ \cite{Range90}, smooth bounded finite type 
convex domains in $\C^n$ \cite{DiederichFischerFornaess99}, and 
some infinite type smooth bounded convex domains in $\C^2$  
\cite{FornaessLeeYuan11}.

Given a holomorphic mapping $f:\D\to \C^m$ 
(where $\D\subset \C^n$) and $\lambda\in \C^m$, we denote  the union 
of all non-isolated points of $f^{-1}(\lambda)$ by $\D_{f,\lambda}$. Since 
$f^{-1}(\lambda)$ is a complex subvariety of $\D$ (for $\lambda$ in the 
range of $f$), it follows that $\D_{f,\lambda}$ is the union of all positive 
dimensional connected components of $f^{-1}(\lambda)$. In the case 
$f$ extends smoothly up to the boundary of $\D,$ we define 
$\D'_{f,\lambda}$ to be the union of all non-isolated points of 
$f^{-1}(\lambda)$ within $\Dc$.  Clearly 
$\D'_{f,\lambda}\subset \D_{f,\lambda}\cup b\D$ where $b\D$ denotes 
the boundary  of $\D$. We define  
 \[\D_f=\bigcup_{\lambda\in \C^m}\D_{f,\lambda}.\]
It is clear that $\Omega_f$ is a subset of the set where the Jacobian 
of $f$ has rank strictly less than $n$. 

Now we are ready to present our first approximation result.

\begin{theorem}\label{Thm1}
Let $\D$ be a bounded $L^{\infty}$-pseudoconvex domain in 
$\C^n$ and $f_j\in H^{\infty}(\D)$ for $j=1,\ldots,m$. Assume that  
$g\in C(\overline{\D})$ such that $g|_{b\D\cup \Omega_f}=0$ where 
$f=(f_1,\ldots,f_m)$. Then $g$ belongs to the closure of 
 $H^{\infty}(\D)[\overline{f_1},\cdots, \overline{f_m}]$ in $L^{\infty}(\D)$. 
\end{theorem}

Theorem \ref{Thm1} and \cite[Theorem 4.2]{IzzoLi13} lead to the 
following corollary.

\begin{corollary}\label{CorLp}
Let $\D$ be a bounded $L^{\infty}$-pseudoconvex domain in $\C^n$ 
and $f_j\in H^{\infty}(\D)$ for $j=1,\ldots, m$ and $n\leq m$. 
Then the following are equivalent.
\begin{itemize}
\item [i.] $H^{\infty}(\D)[\overline{f_1},\ldots, \overline{f_m}]$ is dense in 
$L^p(\D)$ for all $0< p<\infty$,
\item [ii.] $H^{\infty}(\D)[\overline{f_1},\ldots, \overline{f_m}]$ is dense in 
$L^p(\D)$ for some $1\leq p<\infty$,
\item [iii.] the Jacobian of 
$f=(f_1,\ldots, f_m)$ has rank $n$  for some $z\in \D.$ 	
\end{itemize}
\end{corollary}

To formulate our next result we will need the following notation. The 
set of holomorphic functions on $\D$ that have smooth extensions up 
to the boundary is denoted by $A^{\infty}(\D)$. Given a compact set 
$K\subset \overline{\D},$ we will denote by $A_{\Dc}(K)$ the norm closed 
subalgebra of continuous functions on $K$ spanned by restrictions of 
$A^{\infty}(U\cap \Omega)$ onto $K$, where $U$ runs through open 
neighborhoods of $K.$

\begin{theorem}\label{Thm2}
Let $\D$ be a smooth bounded pseudoconvex domain in $\C^n$ and  
$f_j\in A^{\infty}(\D)$ for $j=1,\ldots,m$. Then $g\in C(\Dc)$ belongs to 
the closure of $A^{\infty}(\D)[\overline{f_1},\cdots, \overline{f_m}]$  in 
$L^{\infty}(\D)$ if and only if for any $\lambda$ in the range  of 
$f=(f_1,\ldots,f_m)$ we have 
$g|_{\D'_{f,\lambda}}\in A_{\Dc}(\D'_{f,\lambda})$.
\end{theorem}

Alexander Izzo in \cite[Theorem 1.3]{Izzo11} proved (among other things) 
the following interesting result. 
\begin{theorem*}[Izzo]
Let $A$ be a uniform algebra on a compact Hausdorff space $X$ whose 
maximal ideal space is $X$ and $E\subset X$ be a closed subset such that 
$X\setminus E$ is an $m$-dimensional manifold. Assume that 
\begin{itemize}
\item[i.] for any $p\in X\setminus E$ there exists $f_1,\cdots, f_m\in A$ 
that are $C^1$-smooth on $X\setminus E$ and 
$df_1\wedge\cdots\wedge df_m(p)\neq 0,$
\item[ii.] the functions in $A$ that are  $C^1$-smooth on $X\setminus E$ 
separate points on $X.$
\end{itemize}  
Then $A=\lbrace g\in C(X):g|E\in A|E\rbrace.$ 	
\end{theorem*} 

As pointed out to us by Alexander Izzo, a result along the lines of Theorem \ref{Thm1} 
(for a similar class of domains) can be obtained from \cite{Izzo11} as follows. 
Let us take $X$ to be the maximal ideal space (spectrum) of 
$H^{\infty}(\Omega)$ and $X\setminus E$ to be the set of points in 
$\Omega$ where the Jacobian of $f$ has rank $n$ with $A$ being 
the closure of $H^{\infty}(\D)[\overline{f_1},\cdots, \overline{f_m}]$. 
Then one obtains  Theorem \ref{Thm1}  if the set  $\D_f$   is replaced 
by the set of points  where $J_f$, the Jacobian of $f$, has rank strictly  
less than $n$ (usually a larger set than $\D_f$). 


Next we will present our generalization of the Axler-\v{C}u\v{c}kovi\'c-Rao 
Theorem to $\C^n$, but first we will state the commuting problem for 
Toeplitz operators. 

Let $A^2(\D)$ denote the space of square integrable holomorphic 
functions on $\D$ and $P:L^2(\D)\to A^2(\D)$ be the Bergman projection, 
the orthogonal projection onto $A^2(\D)$. For $g\in L^{\infty}(\D)$, 
the Toeplitz operator $T_g:A^2(\D)\to A^2(\D)$ is defined as 
$T_gf=P(gf)$ for all $f\in A^2(\D)$. 

The \textit{commuting problem} can be stated as follows: Let  
$\phi$ be a non-constant bounded function on $\D$. Determine all 
$\psi \in L^{\infty}(\D)$ such that $[T_{\phi}, T_{\psi}] = 0$.

The commuting problem was solved by Arlen Brown and Paul Halmos on the Hardy 
space of the unit disc in a famous paper \cite{BrownHalmos64}. However, on 
the Bergman space, the problem is still open. Many partial answers has been 
obtained over the years. To list a few, we refer the reader to 
\cite{AxlerCuckovic91,CuckovicRao98,AxlerCuckovicRao00, LeTikaradze17} for 
results over the unit disc; to \cite{Zheng98,Le08,Le17} for results over the ball 
in $\C^n$;  and to \cite{BauerLe11,ChoeYang14,AppuhamyLe16} for results on 
Fock spaces. 

In this paper, we want to highlight the following result of Sheldon Axler, 
\v{Z}eljko \v{C}u\v{c}kovi\'c, and Nagisetti Rao (see \cite{AxlerCuckovicRao00}). 
	
\begin{theorem*}[Axler-\v{C}u\v{c}kovi\'c-Rao]
Let $\D$ be a bounded domain in $\C$ and $\phi$ be a nonconstant 
bounded holomorphic  function on $\D$. Assume that $\psi$ is a 
bounded measurable function on $\D$ such that $T_{\phi}$  and 
$T_{\psi}$ commute. Then $\psi$ is holomorphic.
\end{theorem*}

As an application of our results, we get the following generalization 
of the Axler-\v{C}u\v{c}kovi\'c-Rao Theorem. 

\begin{corollary}\label{CorACR}
Let $\D$ be a bounded $L^{\infty}$-pseudoconvex domain in $\C^n$, 
$g\in L^{\infty}(\D)$, and  $f_j\in H^{\infty}(\D)$ for $j=1,\ldots, m$ 
and $n\leq m$. Assume that the Jacobian of the function 
$f=(f_1, \ldots, f_m):\D\to \mathbb{C}^m$ has rank $n$ for some 
$z\in\D$ and  $T_g$ commutes with $T_{f_j}$ for $1\leq j\leq m$. 
Then $g$ is holomorphic.
\end{corollary}

This paper is organized as follows: The next section contains relevant 
basic facts and results about $\dbar$-Koszul complex. Then we will 
present the proofs of Theorems \ref{Thm1} and \ref{Thm2}. We will 
finish the paper with the proof of  Corollaries \ref{CorLp} and \ref{CorACR}.

\section*{The $\overline{\partial}$-Koszul Complex}

Let $\D$ be a domain in $\C^n$ and $V$ be a vector space of dimension 
$m$ with a basis $\{e_1,e_2,\ldots,e_m\}$. We define 
\[\wedge^rV =\text{span}\left\{e_{j_1}\wedge e_{j_2}\wedge\cdots\wedge 
e_{j_r}:j_1<j_2<\cdots<j_r\right\}\] 
and $\Gamma^{\infty}_{(r,s)}=\wedge^rV \otimes CL^{\infty}_{(0,s)}(\D)$ 
where $r$ and $s$ are nonnegative integers. We note that throughout 
the paper we use the convention that  $\Gamma^{\infty}_{(r,s)}=\{0\}$ if 
$r\geq m+1$ or $s\geq n+1$. Finally, $CL^{\infty}_{(0,0)}(\D)=CL^{\infty}(\D)$.  

We define the  unbounded  operator 
$\dbar: \Gamma^{\infty}_{(r,s)}\to \Gamma^{\infty}_{(r,s+1)}$ as 
$\dbar(e_J\otimes W)=e_J\otimes \dbar W$ where $e_J\in \wedge^r V$ 
and $W\in CL^{\infty}_{(0,s)}(\D)$. The operator $\dbar$ is defined on 
\[Dom_{\infty}(\dbar)=\left\{f\in \Gamma^{\infty}_{(r,s)}:\dbar f\in 
\Gamma^{\infty}_{(r,s+1)}\right\}.\] 
Let $f=(f_1,\ldots,f_m):\D\to \C^m$ be a bounded holomorphic mapping. 
Then for $0\leq s\leq n$ and $0\leq r\leq m$ we define the operator  
\[\mathcal{T}_f:\Gamma^{\infty}_{(r+1,s)}\to \Gamma^{\infty}_{(r,s)}\]  
with the following properties: 
\begin{enumerate}
\item $\mathcal{T}_f(e_j\otimes W)=f_jW$,
\item $\mathcal{T}_f(A\wedge B)=\mathcal{T}_f(A)\wedge B+(-1)^{|A|_1}A\wedge 
\mathcal{T}_fB$ (here $|.|_1$ is the order of $A$ in $\cup_{r=0}^m\Lambda^rV$),  
\item $\mathcal{T}_f \dbar=\dbar \mathcal{T}_f$ on $Dom_{\infty}(\dbar)$ 
for $0\leq s\leq n$ and $0\leq r\leq m$,
 \item $\mathcal{T}_f\mathcal{T}_f=0$ and $\dbar\dbar=0$.
 \end{enumerate}

We note that  $\mathcal{T}_fW=0$ for $W\in \Gamma^{\infty}_{(0,s)}$ 
and $0\leq s\leq n$. 

\begin{lemma} \label{Lem1} 
Let $\D$ be a bounded domain in $\C^n, 0\leq s\leq n,0\leq r\leq m$, 
and  $f=(f_1,\ldots,f_m):\D\to \C^m$ be a bounded holomorphic 
mapping. Assume that $W\in \Gamma^{\infty}_{(r,s)}$ such that 
$\text{supp}(W)\subset  \D$ and $\text{supp}(W)\cap f^{-1}(0)=\emptyset$.
\begin{itemize}
 \item [i.]  If $\mathcal{T}_fW=0$, then 
there exists  $Y\in \Gamma^{\infty}_{(r+1,s)}$ such that 
\begin{itemize}
\item[a.] $\mathcal{T}_fY=W$,
\item[b.] $\text{supp}(Y)\subset \D$ and $\text{supp}(Y)\cap f^{-1}(0)=\emptyset$.
\end{itemize}
\item[ii.] If $\mathcal{T}_fW=0$ and  $\dbar W\in \Gamma^{\infty}_{(r,s+1)}$,  
then there exists  $Y\in \Gamma^{\infty}_{(r+1,s)}$ such that 
\begin{itemize}
\item[a.] $\dbar Y\in \Gamma^{\infty}_{(r+1,s+1)}$ and $\mathcal{T}_fY=W$, 
\item[b.] $\text{supp}(Y)\subset \D$ and $\text{supp}(Y)\cap f^{-1}(0)=\emptyset$.
\end{itemize}
\end{itemize}
\end{lemma}

\begin{proof}
First let us prove the lemma in case $r=m$.  In this case one can 
show that  $\mathcal{T}_fW=0$ and $\text{supp}(W)\cap f^{-1}(0)=\emptyset$ 
imply that $W=0$. So we can choose $Y=0\in  \Gamma^{\infty}_{(m+1,s)}$. 
For the rest of the proof we will assume that $0\leq r\leq m-1$. 

Now let us prove i.  Let $\chi\in C^{\infty}_0(\D)$ be a smooth compactly 
supported cut-off  function such that $\chi=1$ on a neighborhood of 
$\text{supp}(W)$ and  $\text{supp}(\chi)\cap f^{-1}(0)=\emptyset$. 
We define   
\[g_j=\frac{\chi \overline{f_j}}{\sum_{l=1}^{m}|f_l|^2}\]
and 
\[X=\sum_{j=1}^m e_j\otimes g_j\in\Gamma^{\infty}_{(1,0)}.\]
Then $g_j\in C^{\infty}_0(\D)$ for $j=1,2,\ldots,m$ and 
$\mathcal{T}_fX=1\in\Gamma^{\infty}_{(0,0)}$  on the support of $W$ 
because $\chi=1$ on a neighborhood of $\text{supp}(W)$ and  
$\sum_{j=1}^{m}f_j(z)g_j(z)=1$ whenever $\chi(z) =1$. 

Let us define $Y=X\wedge W\in \Gamma^{\infty}_{(r+1,s)}$. Then 
$\text{supp}(Y)$ is a compact  subset of $\D$ and 
$\text{supp}(Y)\cap f^{-1}(0)=\emptyset$. Furthermore,  
$\mathcal{T}_fX=1$ on the support of $W$ and  
\[\mathcal{T}_fY=\mathcal{T}_f(X)\wedge W-X\wedge 
\mathcal{T}_fW=1\wedge W=W \]
because $\mathcal{T}_fW=0$. 

To prove ii. we observe that, in the proof of i. above, $X$ is smooth 
compactly supported in $\D$. Therefore, if $\dbar W$ is bounded 
then so is $\dbar Y$ as $Y=X\wedge W$. 
\end{proof}

If $f_j\in A^{\infty}(\D)$ for $j=1,2,\ldots,m$ in the lemma above, 
 we have the following lemma. 

\begin{lemma} \label{Lem1.1} 
Let $\D$ be a bounded domain in $\C^n,V$ be an $m$-dimensional 
vector space, and $f_j\in A^{\infty}(\D)$ for $j=1,2,\ldots,m$.  
Assume that $W\in \wedge^rV \otimes C^{\infty}_{(0,s)}(\Dc)$ for 
$0\leq r\leq m, 0\leq s\leq n$, and $\text{supp}(W)\cap f^{-1}(0)=\emptyset$ 
where $f=(f_1,\ldots,f_m)$. If $\mathcal{T}_fW=0$ then there exists 
$Y\in \wedge^{r+1}V \otimes C^{\infty}_{(0,s)}(\Dc)$ such that 	
$\text{supp}(Y)\cap f^{-1}(0)=\emptyset$ and $\mathcal{T}_fY=W$.		
\end{lemma}
 \begin{proof}
The proof of this lemma is very similar to the proof of Lemma \ref{Lem1}. 
 The only difference is that we choose $\chi\in C^{\infty}(\Dc)$ be a 
 smooth function such that $\chi=1$ on a neighborhood of 
 $\text{supp}(W)$ and $\text{supp}(\chi)\cap f^{-1}(0)=\emptyset$.  
 \end{proof}
 
\begin{lemma} \label{Lem2}
Let $\D$ be a bounded $L^{\infty}$-pseudoconvex domain in 
$\C^n,f=(f_1,\ldots,f_m):\D\to \C^m$  be a bounded holomorphic 
mapping, and $W\in \Gamma^{\infty}_{(r,s)}$ for  $0\leq r\leq m$ and 
$1\leq s\leq n$ such that 
\begin{itemize}
\item[i.] $\text{supp}(W)\subset \D$ and 
	$\text{supp}(W)\cap f^{-1}(0)=\emptyset$, 
\item[ii.] $\dbar W=0$ and $\mathcal{T}_fW=0$. 
\end{itemize}
Then there exists $Y\in \Gamma^{\infty}_{(r+1,s-1)}$ such that  
$Y\in Dom_{\infty}(\dbar)$ and $\mathcal{T}_f\dbar Y=W$.
\end{lemma}

\begin{proof}
In case $r=m$, as in the proof of Lemma \ref{Lem1}, one can 
show that if $W$ satisfies the conditions of the lemma then 
$W=0$. So we can choose $Y=0$. For the rest of the proof 
we will assume that $0\leq r\leq m-1$.
    
First we will assume that  $\D$ is a bounded $L^{\infty}$-pseudoconvex 
domain.  We will use a descending induction on $s$ to prove this lemma.  
So let $s=n, 0\leq r\leq m-1$, and 
 $W\in \Gamma^{\infty}_{(r,n)}$ such that $\text{supp}(W)\subset 
 \D,\text{supp}(W)\cap f^{-1}(0)=\emptyset$, and $\mathcal{T}_fW=0$ 
 ($\dbar W=0$ as any $(0,n)$-form is $\dbar$-closed). Then i. in 
 Lemma \ref{Lem1} implies that there exists 
 $Y_1\in \Gamma^{\infty}_{(r+1,n)}$ with  the following properties: 
\begin{itemize}
 \item[i.] $\text{supp}(Y_1)\subset \D$ and 
 	$\text{supp}(Y_1)\cap f^{-1}(0)=\emptyset$, 
 \item[ii.] $\mathcal{T}_fY_1=W$.
\end{itemize} Furthermore, since $Y_1\in \Gamma^{\infty}_{(r+1,n)}$ it 
is $\dbar$-closed. Then (since $\D$ is $L^{\infty}$-pseudoconvex) there exists 
$Y\in\Gamma^{\infty}_{(r+1,n-1)}$ 
such that $\dbar Y=Y_1$. That is, $\mathcal{T}_f\dbar Y=W$.  

Now we will assume that the lemma is true for $s=k+1,k+2, \ldots, n$ 
and  $r=0,1,\ldots,m-1$. Let $0\leq r\leq m-1$ and assume that  
$W\in \Gamma^{\infty}_{(r,k)}$ with the following properties: 
\begin{itemize}
\item[i.] $\text{supp}(W)\subset \D$ and 
	$\text{supp}(W)\cap f^{-1}(0) =\emptyset$, 
\item[ii.] $\dbar  W=0$ and $ \mathcal{T}_fW=0$.
\end{itemize}
Then ii. in Lemma \ref{Lem1} implies that there exists  
$Y_1\in \Gamma^{\infty}_{(r+1,k)}$ such that
\begin{itemize}
\item[i.] $\dbar Y_1\in \Gamma^{\infty}_{(r+1,k+1)}$ 
	and $W=\mathcal{T}_fY_1$, 
\item[ii.] $\text{supp}(Y_1)\subset \D$ and 
 	$\text{supp}(Y_1)\cap f^{-1}(0)=\emptyset$.
\end{itemize}
Then 
\[\mathcal{T}_f\dbar Y_1=\dbar \mathcal{T}_fY_1=\dbar W=0.\]
So $\dbar Y_1$ satisfies the conditions in the lemma for $s=k+1$. 
That is,  $\dbar Y_1\in \Gamma^{\infty}_{(r+1,k+1)}$ such that 
\begin{itemize}
\item[i.] $\text{supp}(\dbar Y_1)\subset \D$ and 
	$\text{supp}(\dbar Y_1)\cap f^{-1}(0)=\emptyset$, 
\item[ii.]  $\dbar \dbar Y_1=0$  and $\mathcal{T}_f\dbar Y_1=\dbar W=0$. 
\end{itemize}
By the induction hypothesis, there exists $Y_2\in \Gamma^{\infty}_{(r+2,k)}$ 
such that  $\dbar Y_2\in \Gamma^{\infty}_{(r+2,k+1)}$ and 
$\mathcal{T}_f\dbar Y_2=\dbar Y_1$. Then 
\[\dbar \mathcal{T}_f Y_2=\mathcal{T}_f\dbar Y_2=\dbar Y_1.\]
We define $Y_3=Y_1-\mathcal{T}_fY_2\in \Gamma^{\infty}_{(r+1,k)}$.
Then the equality above implies that  
\[\mathcal{T}_fY_3=\mathcal{T}_f Y_1-\mathcal{T}_f \mathcal{T}_fY_2=W\]
and $\dbar Y_3=\dbar Y_1-\dbar \mathcal{T}_fY_2=0.$
 Since $\D$ is $L^{\infty}$-pseudoconvex domain we conclude that  there 
 exists $Y\in \Gamma^{\infty}_{(r+1,k-1)}$ such that $\dbar Y=Y_3$. That is, 
$\mathcal{T}_f\dbar Y=W$. Hence the proof of Lemma \ref{Lem2} is complete. 
\end{proof}

\begin{lemma}\label{Lem2.1}
Let $\Omega$ be a smooth bounded pseudoconvex domain in $\C^n, V$ 
be an $m$-dimensional vector space,  and  $f_i\in A^{\infty}(\Omega)$ 
for $i=1,\ldots,m$. Assume that
$W\in \wedge^{r}V \otimes C^{\infty}_{(0,s)}(\Dc)$
for $0\leq r\leq m$ and $1\leq s\leq n$ such that 
$\text{supp}(W)\cap f^{-1}(0)=\emptyset, \dbar W=0$, and 
$\mathcal{T}_fW=0$. Then there exists  
$Y\in \wedge^{r+1}V \otimes C^{\infty}_{(0,s-1)}(\Dc)$ such that  
$\mathcal{T}_f\dbar Y=W$.
\end{lemma}
\begin{proof}
This proof is similar to the proof of Lemma \ref{Lem2} with the 
following changes: Instead of Lemma \ref{Lem1} we use 	
Lemma \ref{Lem1.1} and, at the last step (since and 
$f_j\in A^{\infty}(\D)$), we use the following result of Joseph Kohn 
\cite{Kohn73} (see also \cite[Theorem 6.1.1]{ChenShawBook}): 
Let $\D$ be a smooth bounded pseudoconvex domain in 
$\C^n, 1\leq q\leq n$, and  $u\in C^{\infty}_{(0,q)}(\Dc)$ with 
$\dbar u=0$. Then there exists $f\in C^{\infty}_{(0,q-1)}(\Dc)$ 
such that $\dbar f=u$.
\end{proof}	

\begin{lemma}\label{Lem1'}
Let $\D$ be a bounded domain in $\C^n$ and $f_j\in H^{\infty}(\D)$ 
for $j=1,\ldots, m$ such that $\sum_{j=1}^m|f_j|^2>\ep$ on  $\D$ 
for some $\ep>0$ and $\partial f_j\in L^{\infty}_{(1,0)}(\D)$ for 
$j=1,\ldots,m$. Assume that $W\in \Gamma^{\infty}_{(r,s)}$ for 
$0\leq r\leq m$ and $0\leq s\leq n$ such that 	$\mathcal{T}_fW=0$ 
and  $\dbar W\in \Gamma^{\infty}_{(r,s+1)}$. Then there exists 
$Y\in \Gamma^{\infty}_{(r+1,s)}$ such that 
$\dbar Y\in \Gamma^{\infty}_{(r+1,s+1)}$ and  $\mathcal{T}_fY=W$.
\end{lemma}
\begin{proof} 
The proof will 	be similar to the proof of Lemma \ref{Lem1}.
Let $V$ be a vector space of dimension $m$ and 
$\{e_1,e_2,\ldots,e_m\}$ be a basis for $V$.  We define 
\[g_j=\frac{\overline{f_j}}{\sum_{l=1}^m|f_l|^2}\]
and $X=\sum_{j=1}^me_j\otimes g_j\in\Gamma^{\infty}_{(1,0)}$. 
Then $g_j\in L^{\infty}(\D)$ and 
\[\dbar g_j=  \frac{\dbar \overline{f_j}}{\sum_{l=1}^m|f_l|^2}
 	-\frac{\overline{f_j}\sum_{l=1}^mf_l
	\dbar\overline{f_l} }{\left(\sum_{l=1}^m|f_l|^2\right)^2} 
 	\in L^{\infty}_{(0,1)}(\D).\]
Furthermore, 
$\dbar X=\sum_{j=1}^m e_j\otimes\dbar g_j\in\Gamma^{\infty}_{(1,1)}$. 
Then $Y=X\wedge W\in \Gamma^{\infty}_{(r+1,s)}$ satisfies the following properties: 
$\dbar Y= \dbar X\wedge W+X\wedge \dbar W \in \Gamma^{\infty}_{(r+1,s+1)}$ and 
\[\mathcal{T}_fY=\mathcal{T}_f(X)\wedge W-X\wedge 
\mathcal{T}_fW=1\wedge W=W \]
as $\mathcal{T}_fW=0$. 
\end{proof}

\begin{proposition}\label{Prop1}
Let $\D$ be a bounded $L^{\infty}$-pseudoconvex domain in $\C^n$ 
and  $f_j\in H^{\infty}(\D)$ for $j=1,\ldots, m$ such that 
$\sum_{j=1}^m|f_j|^2>\ep$ on $\D$ for some $\ep>0$ and 
$\partial f_j\in L^{\infty}_{(1,0)}(\D)$ for $j=1,\ldots,m$. Assume 
that $W\in \Gamma^{\infty}_{(r,s)}$ for $0\leq r\leq m$ and 
$0\leq s\leq n$ such that $\dbar W=0$ and $\mathcal{T}_fW=0$. 
Then there exists  $Y\in \Gamma^{\infty}_{(r+1,s)}$ such that 
$\dbar Y=0$ and  $\mathcal{T}_fY=W$.
\end{proposition}
\begin{proof}
We will use a descending induction on $s$ as in the proof of 
Proposition \ref{Prop1}. Let $s=n$. Any form of type $(r,n)$ 
for $0\leq r\leq m$ is $\dbar$-closed. Then $\dbar Y=0$ and  
Lemma \ref{Lem1'} implies that there exists 
$Y\in \Gamma^{\infty}_{(r+1,n)}$ such that  $\mathcal{T}_fY=W$.  

Now we will assume that the lemma is true for $s=l+1,l+2, \ldots, n$ 
and $r=0,1,\ldots,m$ to prove that it is also true for $s=l\leq n-1$ 
and $0\leq r\leq m$.

Assume that  $W\in \Gamma^{\infty}_{(r,l)}$ such that  
$\dbar  W=0$ and $ \mathcal{T}_fW=0$. Then 
Lemma \ref{Lem1'} implies that there exists  
$\widetilde{Y}\in \Gamma^{\infty}_{(r+1,l)}$ such that
$\dbar \widetilde{Y}\in \Gamma^{\infty}_{(r+1,l+1)}$ and 
$W=\mathcal{T}_f\widetilde{Y}$. Then 
\[\mathcal{T}_f\dbar \widetilde{Y}
=\dbar \mathcal{T}_f\widetilde{Y}=\dbar W=0.\]
So $\dbar \widetilde{Y}$ satisfies the 
conditions in the lemma for $s=l+1$. That is, 
$\dbar \widetilde{Y}\in \Gamma^{\infty}_{(r+1,l+1)}, 
	\dbar \dbar \widetilde{Y}=0$  
and  $\mathcal{T}_f\dbar \widetilde{Y}=\dbar W=0$.  Then, by 
the induction hypothesis, there exists $Y_1\in \Gamma^{\infty}_{(r+2,l+1)}$ 
such that $\dbar Y_1=0$ and $\mathcal{T}_fY_1=\dbar \widetilde{Y}$. 
Then since $\D$ is a $L^{\infty}$-pseudoconvex domain 
there exists $Y_2\in \Gamma^{\infty}_{(r+2,l)}$ 
such that $\dbar Y_2=Y_1$. Then 
\[\dbar \mathcal{T}_f Y_2=\mathcal{T}_f\dbar Y_2
=\mathcal{T}_fY_1=\dbar \widetilde{Y}.\]
We define $Y=\widetilde{Y}-\mathcal{T}_fY_2\in \Gamma^{\infty}_{(r+1,l)}$.
Then the equality above implies 
that  $\dbar Y=\dbar \widetilde{Y}-\dbar \mathcal{T}_fY_2=0$ and 
\[\mathcal{T}_fY 
	=\mathcal{T}_f \widetilde{Y}-\mathcal{T}_f\mathcal{T}_fY_2=W.\]
Hence the proof of Proposition \ref{Prop1} is complete. 
\end{proof}

As a corollary to the previous proposition (with $W=1$ and $r=s=0$) 
we get the following Corona type result.  We refer the reader to 
\cite{Krantz14} and the references therein for more information 
about Corona problem on domains in $\C^n$.

\begin{corollary}\label{CorCorona}
Let $\D$ be a bounded $L^{\infty}$-pseudoconvex domain in $\C^n$ 
and $f_j\in H^{\infty}(\D)$ for $j=1,\ldots, m$ such that 
$\sum_{j=1}^m|f_j|^2>\ep$ on $\D$ for some $\ep>0$ and 
$\partial f_j\in L^{\infty}_{(1,0)}(\D)$ for $j=1,\ldots,m$. Then 
there exists $g_i\in H^{\infty}(\D)$ for $j=1,\ldots, m$ such that 
$\sum_{j=1}^mf_jg_j=1$. 
\end{corollary}

\section*{Proofs of Results}
The proofs of the theorems  are mainly inspired by  the proof  in  
Christopher Bishop's paper \cite{Bishop89}.
\begin{proof}[Proofs of Theorems \ref{Thm1} and \ref{Thm2}]
The proofs of both theorems are very similar. So we will present the 
proof of Theorem \ref{Thm1} and comment on how the proof of 
Theorem \ref{Thm2} differs as we go along.  

Let $\epsilon>0$ and $\lambda\in \C^m$. Since $g\in C(\overline{\D})$ 
and $g|_{b\D\cup \Omega_f}=0$, there exist	
$g^{\lambda}\in C^{\infty}(\overline{\D})$  such that 
\begin{itemize}
\item[i.] $\sup\{|g(z)-g^{\lambda}(z)|:z\in \Dc\}<\ep$, 
\item[ii] $\text{supp}(\dbar g^{\lambda})\cap 
 (b\D \cup f^{-1}(\lambda)) =\emptyset$.
\end{itemize} 
 In the proof Theorem \ref{Thm2} the second condition above 
 is replaced by $\text{supp}(\dbar  g^{\lambda})\cap  f^{-1}(\lambda)=\emptyset$. 
 This can be seen as follows: We choose an open set $U_{\ep}$ in $\C^n$ 
 containing $f^{-1}(\lambda)$ and $g_{\ep}\in A^{\infty}(U_{\ep}\cap \D)$ 
 such that $|g-g_{\ep}|<\ep/2$ on $f^{-1}(\lambda)$. Then we choose 
 $\chi_{\ep}\in C^{\infty}_0(U_{\ep})$ such that, 
 $0\leq \chi_{\ep}\leq 1,\chi_{\ep}=1$ on a neighborhood of 
 $f^{-1}(\lambda)$, and 
 \[\text{supp}(\chi_{\ep})\cap \Dc \subset 
 \left\{z\in U_{\ep}\cap \Dc:|g(z)-g_{\ep}(z)|<\ep\right\}.\]
Then we define 
 $g^{\lambda}=(1-\chi_{\ep})g+\chi_{\ep}g_{\ep}$. Since 
 $g^{\lambda}$ is holomorphic on a neighborhood of $f^{-1}(\lambda)$ 
 we have $\dbar  g^{\lambda}=0$ on the same neighborhood. 
Furthermore, $|g^{\lambda}(z)-g(z)|=\chi_{\ep}(z)|g_{\ep}(z)-g(z)|< \ep$ 
for all $z\in \Dc$.

Using Lemma \ref{Lem2} with $r=0,s=1$, and $W=\dbar g^{\lambda}$ 
we get $Y=\sum_{l=1}^me_l\otimes H_l\in \Gamma^{\infty}_{(1,0)}$ 
such that 
\begin{align}\label{Eqn1}
\dbar g^{\lambda}=\mathcal{T}_{f-\lambda}\dbar Y
=\sum_{l=1}^m (f_l-\lambda_l)\dbar H_l^{\lambda}.
\end{align}
The above equality implies that  
\[G_{\lambda}=g^{\lambda}-\sum_{l=1}^m (f_l-\lambda_l)H_l^{\lambda}\]
is a bounded holomorphic function.  
	
In the proof of Theorem \ref{Thm2}, we use Lemma \ref{Lem2.1} 
and get $H_l^{\lambda}\in C^{\infty}(\overline{\D})$ for $l=1,\ldots,m$ 
in the equation \eqref{Eqn1} and $G_{\lambda}$ is smooth up to the 
boundary. Therefore, for $z\in \D$ we have   
\[|G_{\lambda}(z)-g^{\lambda}(z)|
\leq \sum_{l=1}^m |f_l(z)-\lambda_l| \sum_{s=1}^m|H_s^{\lambda}(z)|.\]
Then the above inequality implies that for   
$M_{\lambda}=\sum_{s=1}^m\|H^{\lambda}_s\|_{L^{\infty}(\D)}<\infty $ 
we have   
\begin{align}\label{Eqn2}
|G_{\lambda}(z)-g^{\lambda}(z)|\leq M_{\lambda}|f(z)-\lambda| 
\end{align}
for $z\in \D$.
	
Compactness of $\overline{f(\D)}$ implies that we can choose a finite 
collection of points  $\{\lambda_j\}_{j=1}^k\subset \overline{f(\D)}$ 
such that $\{B(\lambda^j,\epsilon M^{-1}_{\lambda^j})\}_{j=1}^k$ forms 
a finite open cover for  $\overline{f(\D)}$. Let $\{\chi_j\}_{j=1}^k$ be a 
smooth partition of unity on $\overline{f(\D)}$ such that 
$0\leq \chi_j\leq 1$ and $\text{supp}(\chi_j)\subset U_j$. 
Then $\{f^{-1}(B(\lambda^j,\epsilon M^{-1}_{\lambda^j}))\}_{j=1}^k$ 
is an cover for $\D$ and 
$|f(z)-\lambda^j|< \epsilon M^{-1}_{\lambda^j}$ for 
$z\in f^{-1}(B(\lambda^j,\epsilon M^{-1}_{\lambda^j}))$. Then for $z\in \D$ 
we have 
\begin{align*}
\left|\sum_{j=1}^k G_{\lambda^j}(z)\chi_j(f)(z)-g(z)\right|
\leq& \sum_{j=1}^k |G_{\lambda^j}(z)-g(z)|\chi_j(f(z))\\
\leq & \sum_{j=1}^k|G_{\lambda^j}(z)-g^{\lambda^j}|\chi_j(f(z))
+\sum_{j=1}^k|g^{\lambda^j}(z)-g(z)|\chi_j(f(z))\\
\leq &  \sum_{j=1}^kM_{\lambda^j}|f(z)-\lambda^j|\chi_j(f(z))
	+\ep  \sum_{j=1}^k\chi_j(f(z)) \\
\leq& 2\epsilon.
\end{align*}
Finally, the Stone-Weierstrass Theorem implies that $\chi_j(f)$ can be 
approximated uniformly on $\Dc$ by elements of 
$\mathbb{C}[f_1,\ldots, f_m, \overline{f_1},\ldots ,\overline{f_m}]$. 
Hence the proofs of Theorems \ref{Thm1} and \ref{Thm2} are complete. 
\end{proof}

Hartogs Extension Theorem together Theorem \ref{Thm2} lead to the 
following corollary. 

\begin{corollary}
Let $\D$ be a bounded $L^{\infty}$-pseudoconvex domain in $\C^n$. 
Assume that $f=(f_1,\ldots,f_m):\D\to \C^m$  be a bounded holomorphic 
mapping and $g\in C(\Dc)$ such that $\dbar g$ is supported away from 
$b\D$ and the set of points at which the Jacobian of $f$ has rank strictly 
less than $n$. Then $g$ belongs to the closure of 
$H^{\infty}(\D)[\overline{f_1},\cdots, \overline{f_m}]$ in $L^{\infty}(\D)$. 
\end{corollary}

\begin{proof}
Since $\dbar g$ vanishes near the boundary of $\D$, Hartogs Extension 
Theorem implies that there exists $g_1\in H^{\infty}(\D)$ such that $g=g_1$ 
near the boundary of $\D$. Then $g_2=g-g_1\in C(\Dc)$ and $g_2$ is compactly 
supported in $\D$. Furthermore, $g_2$ is holomorphic on a neighborhood of 
the set where the Jacobian of $f$ has rank strictly less than $n$. Therefore, 
Theorem \ref{Thm2} implies that $g_2$  can be approximated in  the sup-norm by 
functions in  $H^{\infty}(\D)[\overline{f_1}, \ldots, \overline{f_m}]$. 
This completes the proof of the corollary.
\end{proof}

Next we provide the proof of Corollary \ref{CorLp}.

\begin{proof}[Proof of Corollary \ref{CorLp}]
Obviously i. implies ii. So to prove that ii. implies iii., let us assume 
that $H^{\infty}(\D)[\overline{f_1},\ldots, \overline{f_m}]$ is dense 
in $L^p(\D)$ for some $1\leq p<\infty$. Let $B\subset \D$ be a ball 
such that $\overline{B}\subset \D$. Then, the algebra 
$H^{\infty}(B)[\overline{f_1},\ldots, \overline{f_m}]$ is dense 
in $L^p(B)$ for some $1\leq p<\infty$. Moreover, the algebra generated by 
$\{z_1,\ldots,z_n\}$ is dense in $H^{\infty}(B)$ and $f_1,\cdots, f_m$ 
are holomorphic on a neighborhood of $\overline{B}.$ 
Next we adopt \cite[Theorem 4.2]{IzzoLi13} to our set-up. Namely,  
\cite[Theorem 4.2]{IzzoLi13} implies that if the algebra generated by  
$\{z_1,\ldots, z_n,\overline{f}_1,\ldots, \overline{f}_m\}\subset C^{\infty}(B)$ 
is dense in $L^p(B)$ for some $1\leq p<\infty$ then the real Jacobian of 
$\{z_1,\ldots, z_n,\overline{f}_1,\ldots, \overline{f}_m\}$ is of full rank 
on a dense open set in $B$. Hence the rank of $J_f$ is $n$ on a dense 
open subset in $B$ and (by identity principle) in $\D$. Hence, we have iii.

Finally, to prove iii. implies i. we assume that the rank  of $J_f$ is $n$ for some 
$z\in \D$. Then, the set of points at which $J_f$ has rank strictly less than $n$  
is a closed set of measure 0 (see \cite[Theorem 3.7]{RangeBook}). One can 
show that $X_f$, the set of smooth functions with compact support in $\D$  
and vanish where $J_f$ has rank strictly less than $n$, is dense in $L^p(\D)$ 
for all $0<p<\infty$. On the other hand, Theorem \ref{Thm1} implies that 
any function in $X_f$ is in the closure of 
$H^{\infty}(\D)[\overline{f_1},\ldots, \overline{f_m}]$ in $L^{\infty}(\D)$.  
Therefore,   $H^{\infty}(\D)[\overline{f_1},\ldots, \overline{f_m}]$ is dense in 
$L^p(\D).$ Hence, we have i.
\end{proof}

We finally end the paper with the proof of Corollary \ref{CorACR}.

\begin{proof}[Proof of Corollary \ref{CorACR}]
We will use the fact that $T_g$ can be defined by the following formula  
\[\langle T_g\phi, \psi\rangle_{A^2(\D)}=\langle g\phi, \psi\rangle_{L^2(\D)}\]
for all $ \phi, \psi \in A^2(\D).$ Since $T_g$ commutes with $T_{P(f)}$, 
for any holomorphic polynomial $P$, we have
\[\langle gP(f),\psi \rangle= \langle T_gT_{P(f)}(1), \psi\rangle
=\langle P(f)T_g(1),\psi\rangle\] 
for all $\psi\in A^2(\D)$. Then $\langle T_{g}(1)-g, \overline{P(f)}\psi \rangle=0$ 
for all $\psi\in A^2(\D)$. Since, by Corollary \ref{CorLp}, the subspace generated by 
$\{\overline{P(f)}\psi: \psi \in A^2(\D)\}$ is dense in $L^2(\D),$  we conclude 
that  $T_g(1)=g$. That is, $g$ is holomorphic.
\end{proof}

\section*{Acknowledgement}

We would like to thank Alexander Izzo for reading an earlier manuscript of this paper 
and for providing us with valuable comments. We are also thankful to the 
anonymous referee for helpful feedback.


\end{document}